\documentclass[11pt]{article}%
\usepackage{amsmath}
\usepackage{amsfonts}
\usepackage{amssymb}
\usepackage{graphicx}%
\newtheorem{theorem}{Theorem}

\newtheorem{corollary}[theorem]{Corollary}

\newtheorem{definition}[theorem]{Definition}

\newtheorem{lemma}[theorem]{Lemma}

\newtheorem{proposition}[theorem]{Proposition}
\newtheorem{remark}[theorem]{Remark}

\newenvironment{proof}[1][Proof]{\noindent\textbf{#1.} }{\ \rule{0.5em}{0.5em}}
\graphicspath{    {converted_graphics/}    {/}}
\begin{document}

\begin{center}
{\LARGE Folding, Cycles and Chaos in Planar Systems}

\medskip

\centerline{H. SEDAGHAT \footnote{Department of Mathematics, Virginia Commonwealth University Richmond, Virginia, 23284-2014, USA; Email: hsedagha@vcu.edu}}
\end{center}

\vspace{2ex}


\begin{abstract}
A typical planar system of difference equations can be folded or transformed
into a scalar difference equation of order two plus a passive (non-dynamic)
equation. We discuss this method and its application to identifying and
proving the existence as well as nonexistence of cycles and chaos in a number of systems of
rational difference equations with variable coefficients. These include some
systems that converge to autonomous systems and some that do not; e.g.,
systems with periodic coefficients.

\end{abstract}

\bigskip

\section{Introduction}

It is broadly known that in discrete systems periodic and chaotic behavior may
occur for maps of the interval and other one dimensional manifolds. Planar
difference systems, which generalize interval maps to two-dimensions are also
known to have this feature but they are not
as well-understood. It is by no means simple to prove whether a given planar
map has cycles or exhibits chaos. Certain global results e.g., the Sharkovski
ordering of cycles, are not true for planar maps in general (e.g., the
occurrence of a 3-cycle does not imply the existence of any other cycle). There
are comparatively few methods (e.g., Marotto's snap-back repeller criterion in
\cite{Mar}) that are applicable widely to the study of cycles and chaos in
planar systems.

In this paper we use the new method of folding to explore planar systems and
their orbits. This method has been in use (though not by this name) in
different contexts in the literature. Folding linear systems in both
continuous and discrete time is seen in control theory; the \textquotedblleft
controllability canonical form" is precisely the folding of a controllability
matrix into a linear higher order equation, whether in continuous or discrete
time; see, e.g., \cite{Elay}, \cite{Lasl}. In an entirely different line of
research, in \cite{Eich} a variety of nonlinear differential systems
displaying chaotic behavior are studied and classified by converting them to
ordinary differential equations of order 3 that define jerk (or jolt)
functions, i.e., time rates of change of acceleration.

These ideas in control theory and in chaotic differential systems are special
instances of the same concept, namely, folding systems to equations. In
\cite{SF} these and similar notions are unified by means of a new algorithmic
process for folding difference or differential systems to scalar equations.

In the case of planar systems, folding yields a second-order scalar difference
equation whose analysis provides useful information about the orbits of the
original system in cases where standard methods are unavailing. We establish
the existence or nonexistence of cycles and chaos for various rational planar
systems. Further, since in principle folding applies to nonautonomous systems
in the same way that it does to the autonomous ones, time-dependent parameters
are considered in this study. But the systems that we study here exhibit
cycles and chaos even with constant parameters.

\section{Folding difference systems}

The material in this section comes from \cite{SF}. A (recursive, or explicit)
system of two first-order difference equations is typically defined as%
\begin{equation}
\left\{
\begin{array}
[c]{c}%
x_{n+1}=f(n,x_{n},y_{n})\\
y_{n+1}=g(n,x_{n},y_{n})
\end{array}
\right.  \quad n=0,1,2,\ldots\label{s1}%
\end{equation}
where $f,g:\mathbb{N}_{0}\times D\rightarrow S$ are given functions,
$\mathbb{N}_{0}=\{0,1,2,\ldots\}$ is the set of non-negative integers, $S$ a
nonempty set and $D\subset S\times S$. If $S$ is a subset of the set
$\mathbb{R}$ of the real numbers with the usual topology then (\ref{s1}) is a
\textit{planar system}.

An initial point $(x_{0},y_{0})\in D$ generates a (forward) orbit or solution
$\{(x_{n},y_{n})\}$ of (\ref{s1}) in the state-space $S\times S$ through the
iteration of the function
\[
(n,x_{n},y_{n})\rightarrow(n+1,f(n,x_{n},y_{n}),g(n,x_{n},y_{n})):\mathbb{N}%
_{0}\times D\rightarrow\mathbb{N}_{0}\times S\times S
\]
for as long as the points $(x_{n},y_{n})$ remain in $D.$ If (\ref{s1}) is
\textit{autonomous}, i.e., the functions $f,g$ do not depend on the index $n$
then $(x_{n},y_{n})=F^{n}(x_{0},y_{0})$ for every $n$ where $F^{n}$ denotes
the composition of the map $F(u,v)=(f(u,v),g(u,v))$ of $S\times S$ with itself
$n$ times.

A second-order, scalar difference equation in $S$ is defined as%
\begin{equation}
s_{n+2}=\phi(n,s_{n},s_{n+1}),\quad n=0,1,2,\ldots\label{e1}%
\end{equation}
where $\phi:\mathbb{N}_{0}\times D^{\prime}\rightarrow S$ is a given function
and $D^{\prime}\subset S\times S$. A pair of initial values $s_{0},s_{1}\in S$
generates a (forward) solution $\{s_{n}\}$ of (\ref{e1}) in $S$ if
$(s_{0},s_{1})\in D^{\prime}.$ If $\phi(n,u,v)=\phi(u,v)$ is independent of
$n$ then (\ref{e1}) is autonomous.

An equation of type (\ref{e1}) may be \textquotedblleft unfolded" to a system
of type (\ref{s1}) in a standard way; e.g.,%
\begin{equation}
\left\{
\begin{array}
[c]{l}%
s_{n+1}=t_{n}\\
t_{n+1}=\phi(n,s_{n},t_{n})
\end{array}
\right.  \label{s2}%
\end{equation}

In (\ref{s2}) the temporal delay in (\ref{e1}) is converted to an additional
variable in the state space. All solutions of (\ref{e1}) are reproduced from
the solutions of (\ref{s2}) in the form $(s_{n},s_{n+1})=(s_{n},t_{n})$ so in
this sense, higher order equations may be considered to be special types of
systems. In general, (\ref{e1}) may be unfolded in different ways into systems
of two equations and (\ref{s2}) is not unique.

\begin{definition}
\label{smsl}Let $S$ be a nonempty set and consider a function $f:\mathbb{N}%
_{0}\times D\rightarrow S$ where $D\subset S\times S$. Then $f$ is
\textbf{semi-invertible} or \textbf{partially invertible} if there are sets
$M\subset D$, $M^{\prime}\subset S\times S$ and a function $h:\mathbb{N}%
_{0}\times M^{\prime}\rightarrow S$ such that for all $(u,v)\in M$ if
$w=f(n,u,v)$ then $(u,w)\in M^{\prime}$ and $v=h(n,u,w)$ for all
$n\in\mathbb{N}_{0}$.
\end{definition}

The function $h$ above may be called a semi-inversion, or partial inversion of
$f.$ If $f$ is independent of $n$ then $n$ is dropped from the above notation.

Semi-inversion refers more accurately to the \textit{solvability} of the
equation $w-f(n,u,v)=0$ for $v$ which recalls the implicit function theorem
(see \cite{SF}). On the other hand, a substantial class of semi-invertible
functions is supplied (globally) by the following idea.

\begin{definition}
(Separability)\ Let $(G,\ast)$ be a nontrivial group and let $f:\mathbb{N}%
_{0}\times G\times G\rightarrow G$. If there are functions $f_{1}%
,f_{2}:\mathbb{N}_{0}\times G\rightarrow G$ such that
\[
f(n,u,v)=f_{1}(n,u)\ast f_{2}(n,v)
\]
for all $u,v\in G$ and every $n\geq1$ then we say that $f$ is
\textbf{separable} on $G$ and write $f=f_{1}\ast f_{2}$ for short.
\end{definition}

Every affine function $f(n,u,v)=a_{n}u+b_{n}v+c_{n}$ with real parameters
$a_{n},b_{n},c_{n}$ is separable on $\mathbb{R}$ relative to ordinary addition
for all $n$ with, e.g., $f_{1}(n,v)=a_{n}u$ and $f_{2}(n,v)=b_{n}v+c_{n}.$
Similarly, $f(n,u,v)=a_{n}u/v$ is separable on $\mathbb{R}\backslash\{0\}$
relative to ordinary multiplication.

Now, suppose that $f_{2}(n,\cdot)$ is a bijection for every $n$ and
$f_{2}^{-1}(n,\cdot)$ is its inverse for each $n$; i.e., $f_{2}(n,f_{2}%
^{-1}(n,v))=v$ and $f_{2}^{-1}(n,f_{2}(n,v))=v$ for all $v.$ A separable
function $f$ is semi-invertible if the component function $f_{2}(n,\cdot)$ is
a bijection for each fixed $n$, since for every $u,v,w\in G$
\[
w=f_{1}(n,u)\ast f_{2}(n,v)\Rightarrow v=f_{2}^{-1}(n,[f_{1}(n,u)]^{-1}\ast
w)
\]
where map inversion and group inversion, both denoted by $-1,$ are
distinguishable from the context. In this case, an explicit expression for the
semi-inversion $h$ exists globally as%
\begin{equation}
h(n,u,w)=f_{2}^{-1}(n,[f_{1}(n,u)]^{-1}\ast w) \label{spsh}%
\end{equation}
with $M=M^{\prime}=G\times G.$ We summarize this observation as follows.

\begin{proposition}
\label{sep}Let $(G,\ast)$ be a nontrivial group and $f=f_{1}\ast f_{2}$ be
separable$.$ If $f_{2}(n,\cdot)$ is a bijection for each $n$ then $f$ is
semi-invertible on $G\times G$ with a semi-inversion uniquely defined by
(\ref{spsh}).
\end{proposition}

If $a_{n}\not =0$ (or $b_{n}\not =0$) for all $n$ then the separable function
$f(n,u,v)=a_{n}u+b_{n}v+c_{n}$ is semi-invertible as it can readily be solved
for $u$ (or $v$). If $a_{n},b_{n}$ are both zero for infinitely $n$ then $f$
is separable but not semi-invertible for either $u$ or $v$.

Now, suppose that $\{(x_{n},y_{n})\}$ is an orbit of (\ref{s1}) in $D$. If one
of the component functions in (\ref{s1}), say, $f$ is semi-invertible then by
Definition \ref{smsl} there is a set $M\subset D$, a set $M^{\prime}\subset
S\times S$ and a function $h:\mathbb{N}_{0}\times M^{\prime}\rightarrow S$
\ such that if $(x_{n},y_{n})\in M$ then $(x_{n},x_{n+1})=(x_{n}%
,f(n,x_{n},y_{n}))\in M^{\prime}$ and $y_{n}=h(n,x_{n},x_{n+1}).$ Therefore,
\begin{equation}
x_{n+2}=f(n+1,x_{n+1},y_{n+1})=f(n+1,x_{n+1},g(n,x_{n},y_{n}))=f(n+1,x_{n+1}%
,g(n,x_{n},h(n,x_{n},x_{n+1}))) \label{md}%
\end{equation}
and the function%
\begin{equation}
\phi(n,u,w)=f(n+1,w,g(n,u,h(n,u,w))) \label{se2}%
\end{equation}
is defined on $\mathbb{N}_{0}\times M^{\prime}.$ If $\{s_{n}\}$ is the
solution of (\ref{e1}) with initial values $s_{0}=x_{0}$ and $s_{1}%
=x_{1}=f(0,x_{0},y_{0})$ and $\phi$ defined by (\ref{se2}) then
\[
s_{2}=f(1,s_{1},g(0,s_{0},h(0,s_{0},s_{1})))=f(1,x_{1},g(0,x_{0}%
,h(0,x_{0},x_{1})))=f(1,x_{1},g(0,x_{0},y_{0}))=x_{2}%
\]

By induction, $s_{n}=x_{n}$ and thus $h(n,s_{n},s_{n+1})=h(n,x_{n}%
,x_{n+1})=y_{n}.$ It follows that
\begin{equation}
(x_{n},y_{n})=(s_{n},h(n,s_{n},s_{n+1})) \label{xy}%
\end{equation}
i.e., the solution $\{(x_{n},y_{n})\}$ of (\ref{s1}) can be obtained from a
solution $\{s_{n}\}$ of (\ref{e1}) via (\ref{xy}). Thus the following is true.

\begin{theorem}
\label{A}Suppose that $f$ in (\ref{s1}) is semi-invertible with $M,M^{\prime}$
and $h$ as in Definition \ref{smsl}. Then each orbit of (\ref{s1}) in $M$ may
be derived from a solution of (\ref{e1}) via (\ref{xy}) with $\phi$ given by
(\ref{se2}).
\end{theorem}

The following gives a name to the pair of equations that generate the
solutions of (\ref{s1}) in the above theorem.

\begin{definition}
\label{fld}(Folding) The pair of equations%
\begin{align}
s_{n+2}  &  =\phi(n,s_{n},s_{n+1}),\quad\quad\text{(core)}\label{cor}\\
y_{n}  &  =h(n,x_{n},x_{n+1})\qquad\text{(passive)} \label{pas}%
\end{align}
where $\phi$ is defined by (\ref{se2}) is a folding of the system (\ref{s1}).
The initial values of the core equation are determined from the initial point
$(x_{0},y_{0})$ as $s_{0}=x_{0}$, $s_{1}=f(0,x_{0},y_{0}).$
\end{definition}

We call Equation (\ref{pas}) \textit{passive} because it simply evaluates the
function $h$ on a solution of the core equation (\ref{cor}) --no dynamics or
iterations are involved. Also observe that (\ref{s1}) may be considered an
unfolding of the second-order equation (\ref{cor}) that is generally not
equivalent to the standard unfolding (\ref{s2}) of that equation.

If one of the component functions in the system is separable then a global
result is readily obtained from Theorem \ref{A}\ using (\ref{spsh}).

\begin{corollary}
\label{sprd}Let $(G,\ast)$ be a nontrivial group and $f=f_{1}\ast f_{2}$ be
separable on $G\times G$. If $f_{2}(n,\cdot)$ is a bijection for every $n$
then (\ref{s1}) folds to%
\begin{align}
s_{n+2}  &  =f_{1}(n+1,s_{n+1},g(n,s_{n},f_{2}^{-1}(n,[f_{1}(n,s_{n}%
)]^{-1}\ast s_{n+1}))\label{sde1}\\
y_{n}  &  =f_{2}^{-1}(n,[f_{1}(n,s_{n})]^{-1}\ast s_{n+1}) \label{ysp}%
\end{align}
Each orbit $\{(x_{n},y_{n})\}$ of (\ref{s1}) in $G\times G$ is obtained from a
solution $\{s_{n}\}$ of (\ref{sde1}) with the initial values $s_{0}%
=x_{0},\ s_{1}=f_{1}(0,x_{0})\ast f_{2}(0,y_{0}).$
\end{corollary}

The next result is a special case of Corollary \ref{sprd}.

\begin{corollary}
\label{sl}Let $a_{n},b_{n},c_{n}$ be sequences in a ring $R$ with identity and
let $g:\mathbb{N}_{0}\times R\times R\rightarrow R.$ If $b_{n}$ is a unit in
$R$ for all $n$ then the semilinear system
\begin{equation}
\left\{
\begin{array}
[c]{l}%
x_{n+1}=a_{n}x_{n}+b_{n}y_{n}+c_{n}\\
y_{n+1}=g(n,x_{n},y_{n})
\end{array}
\right.  \label{sls}%
\end{equation}
folds to%
\begin{align}
s_{n+2}  &  =c_{n+1}+a_{n+1}s_{n+1}+b_{n+1}g(n,s_{n},b_{n}^{-1}(s_{n+1}%
-a_{n}s_{n}-c_{n}))\label{sde2}\\
y_{n}  &  =b_{n}^{-1}(s_{n+1}-a_{n}s_{n}-c_{n})\nonumber
\end{align}

Each orbit $\{(x_{n},y_{n})\}$ of (\ref{sls}) in $R$ is obtained from a
solution $\{s_{n}\}$ of (\ref{sde2}) with the initial values $s_{0}%
=x_{0},\ s_{1}=a_{0}x_{0}+b_{0}y_{0}+c_{0}$.
\end{corollary}

A natural question after folding a system is whether the qualitative
properties of the solutions of the core equation (\ref{cor}) are shared by the
orbits of (\ref{s1}). The answer clearly depends on the passive equation so
that despite its non-dynamic nature, (\ref{pas}) plays a nontrivial role in
the folding. The next result illustrates this feature and is used in the next section.

\begin{lemma}
\label{per}Assume that the semi-inversion $h$ in (\ref{pas}) has\ period
$p\geq1$, i.e., $p$ is the least positive integer such that
$h(n+p,u,w)=h(n,u,w)$ for all $(n,u,w)\in\mathbb{N}_{0}\times M^{\prime}.$ Let
$\{s_{n}\}$ be a solution of (\ref{cor}) with initial values $s_{0}=x_{0}$,
$s_{1}=f(0,x_{0},y_{0})$.

(a) If $\{s_{n}\}$ is periodic with\ period $q\geq1$ then the corresponding
orbit $\{(x_{n},y_{n})\}$ of (\ref{s1}) is periodic with period equal to the
least common multiple $\operatorname{lcm}(p,q)$.

(b) If $\{s_{n}\}$ is non-periodic then $\{(x_{n},y_{n})\}$ is non-periodic.
\end{lemma}

\begin{proof}
(a) Recall that $x_{n}=s_{n}$ so that the sequence $\{x_{n}\}$ of the
x-components of $\{(x_{n},y_{n})\}$ has period $q.$ Also by (\ref{pas})
\[
y_{n+\operatorname{lcm}(p,q)}=h(n+\operatorname{lcm}%
(p,q),x_{n+\operatorname{lcm}(p,q)},x_{n+1+\operatorname{lcm}(p,q)}%
)=h(n,x_{n},x_{n+1})=y_{n}%
\]
since both $p$ and $q$ divide $\operatorname{lcm}(p,q).$ Therefore, the
sequence $\{y_{n}\}$ of the y-components of $\{(x_{n},y_{n})\}$ has period
$\operatorname{lcm}(p,q)$ and it follows that $\{(x_{n},y_{n})\}$ has period
$\operatorname{lcm}(p,q)$.

(b) If $\{(x_{n},y_{n})\}$ is periodic then so is $\{x_{n}\},$ which implies
that $\{s_{n}\}$ is periodic.
\end{proof}

\section{Cycles and chaos in a rational system}

Various definitions of chaos for nonautonoumous systems exist in the
literature. Possibly the most familiar form of deterministic chaos, in the
sense of Li and Yorke, is defined generally as follows.

\begin{definition}
\label{chs}(Li-Yorke Chaos) Let $F_{n}:(X,d)\rightarrow(X,d)$ be functions on
a metric space for all $n\geq0$ and define $F_{0}^{n}=F_{n}\circ F_{n-1}%
\circ\cdots\circ F_{0}$ i.e., the composition of maps $F_{0}$ through $F_{n}.$
The nonautonomous system $(X,F_{n})$ is chaotic if there is an uncountable set
$S\subset X$ (the scrambled set) such that for every pair of points $x,y\in
S$,%
\[
\limsup_{n\rightarrow\infty}d(F_{0}^{n}(x),F_{0}^{n}(y))>0,\qquad
\liminf_{n\rightarrow\infty}d(F_{0}^{n}(x),F_{0}^{n}(y))=0
\]

\end{definition}

For planar maps, $F_{n}(u,v)=(f(n,u,v),g(n,u,v))$ on the Euclidean space
$\mathbb{R}^{2}=\mathbb{R}\times\mathbb{R}$. Despite the similarity of the
above definition to the familiar one for interval maps (autonomous
one-dimensional systems) proving that a particular nonautonomous system is
chaotic in the sense of Definition \ref{chs} is a nontrivial task. For a
continuous interval map the existence of a 3-cycle is sufficient for the
occurrence of Li-Yorke chaos \cite{LY} and for a continuously differentiable
map of $\mathbb{R}^{N}$ a sufficient condition is the existence of a snap-back
repeller \cite{Mar}. To take advantage of such relatively practical results,
we may consider nonautonomous systems that are tied in some way to an
autonomous one.

One natural case that is frequently studied in the literature concerns
nonautonomous systems where the sequence $\{F_{n}\}$ converges uniformly to a
function $F$ on $X$ so that $(X,F)$ is an autonomous system; see, e.g.,
\cite{BO} and \cite{can} for studies of pertinent issues, including whether
the occurrence of chaos in the autonomous system implies, or is implied by the
same for the nonautonomous one.

We consider a different approach where a nonautonomous system is tied to an
autonomous one through folding. In this section, we study a rational system
that folds to an autonomous, first-order difference equation for its core. In
this case, the nonautonomous system need not converge to an autonomous one;
e.g., the system may have periodic coefficients. The dynamic aspects of the
core equation are not affected by the time-dependent parameters which
influence the orbits of (\ref{s1}) through the passive equation.

Consider the rational system
\begin{subequations}
\label{rh}%
\begin{align}
x_{n+1}  &  =\frac{\alpha_{n}x_{n}+\beta_{n}y_{n}}{A_{n}x_{n}+y_{n}%
}\label{rha}\\
y_{n+1}  &  =\frac{\alpha_{n}^{\prime}x_{n}+\beta_{n}^{\prime}y_{n}}%
{x_{n}+B_{n}y_{n}} \label{rhb}%
\end{align}
where all coefficients are sequences of real numbers. The autonomous version
of the above system, i.e.,
\end{subequations}
\begin{subequations}
\label{arh}%
\begin{align}
x_{n+1}  &  =\frac{\alpha x_{n}+\beta y_{n}}{Ax_{n}+y_{n}}\\
y_{n+1}  &  =\frac{\alpha^{\prime}x_{n}+\beta^{\prime}y_{n}}{x_{n}+By_{n}}%
\end{align}
has been classified as a type (36,36) system in \cite{cklm} when all
coefficients are nonzero (separate number pairs are assigned to special cases
where one or more of the coefficients are zeros). System (\ref{arh}) is
semiconjugate to a first-order rational equation via the substitution of
$r_{n}=x_{n}/y_{n}$ (or the reciprocal of this ratio); see \cite{hsb1} for a
study of semiconjugate systems. We note that (\ref{arh}) is also a homogeneous
system--a generalization of the aforementioned type of semiconjugacy exists
for such systems; see \cite{maz}.

A comprehensive study of (\ref{arh}) appears in \cite{hk} for non-negative
coefficients where the positive quadrant of the plane is invariant under the
action of the underlying planar map. By analyzing the one-dimensional
semiconjugate map, the authors show that exactly one of the following
possibilities occurs: (i) every non-negative solution of (\ref{arh}) converges
to a fixed point, or (ii) there is a unique positive 2-cycle and every
non-negative solution of (\ref{arh}) either converges to this 2-cycle or to a
fixed point of the system, or (iii) there exist unbounded solutions.

When all parameters in a rational system are non-negative, the positive
quadrant $[0,\infty)^{2}$ is invariant and the underlying mapping of the
system is continuous. In the absence of singularities, linear-fractional
equations such as those in (\ref{rh}) tend to behave mildly and not exhibit
the type of complex behavior that is often associated with rapid rates of
change. So questions remain about the nature of the orbits of (\ref{arh}) for
a wider range of parameters, including negative coefficients. Does the system
have cycles of period greater than two? Can it exhibit complex, aperiodic behavior?

With negative parameters the occurrence of singularities (discontinuity)
raises significant existence and boundedness issues for orbits. We use folding
to identify special cases where singularities are avoided and some results are
obtained about (\ref{rh}) and similar systems.

To fold (\ref{rh}) we first solve (\ref{rha}) for $y_{n}$ to find%
\end{subequations}
\begin{equation}
y_{n}=\frac{x_{n}(\alpha_{n}-A_{n}x_{n+1})}{x_{n+1}-\beta_{n}}=h(n,x_{n}%
,x_{n+1}) \label{rhyn}%
\end{equation}

Next, using (\ref{md}) and (\ref{rhb}) we obtain the following
\textit{first-order} core equation:%
\begin{equation}
x_{n+2}=\frac{\alpha_{n+1}(A_{n}B_{n}-1)x_{n+1}^{2}+[\alpha_{n+1}(\beta
_{n}-\alpha_{n}B_{n})+\beta_{n+1}(A_{n}\beta_{n}^{\prime}-\alpha_{n}^{\prime
})]x_{n+1}+\beta_{n+1}(\alpha_{n}^{\prime}\beta_{n}-\alpha_{n}\beta
_{n}^{\prime})}{A_{n+1}(A_{n}B_{n}-1)x_{n+1}^{2}+[A_{n+1}(\beta_{n}-\alpha
_{n}B_{n})+(A_{n}\beta_{n}^{\prime}-\alpha_{n}^{\prime})]x_{n+1}+\alpha
_{n}^{\prime}\beta_{n}-\alpha_{n}\beta_{n}^{\prime}} \label{rhc}%
\end{equation}

The first-order nature of this folding and the semiconjugacy of planar mapping of the
autonomous system (\ref{arh}) to a one dimensional map are evidently related.
However, folding does not require knowledge of a semiconjugate relation or even of
whether such a relation exists.

Equation (\ref{rhc}) does not have complex solutions for all choices of
parameters. For instance, if $A_{n}=\beta_{n}=\beta_{n}^{\prime}=0$ for all
$n$ then (\ref{rhc}) reduces to the affine equation%
\begin{equation}
x_{n+2}=\frac{\alpha_{n+1}}{\alpha_{n}^{\prime}}x_{n+1}+\frac{\alpha
_{n+1}\alpha_{n}B_{n}}{\alpha_{n}^{\prime}} \label{ao1}%
\end{equation}
which does not exhibit complex behavior with constant or even periodic parameters.

To assure the existence of cycles and the occurrence of chaos even in the
autonomous case, we consider a different special case where%
\begin{equation}
A_{n}=\alpha_{n}^{\prime}=\beta_{n}=0,\quad\alpha_{n},\beta_{n}^{\prime}%
\not =0\quad\text{for all }n\geq0 \label{rsc}%
\end{equation}

These conditions are not necessary for the occurrence of cycles or chaos but
we show that they are sufficient. If conditions (\ref{rsc}) hold then
(\ref{rhc}) reduces to the quadratic equation%
\begin{equation}
x_{n+2}=\frac{\alpha_{n+1}}{\alpha_{n}\beta_{n}^{\prime}}x_{n+1}^{2}%
+\frac{\alpha_{n+1}B_{n}}{\beta_{n}^{\prime}}x_{n+1} \label{rhc1}%
\end{equation}

To simplify calculations we also assume that there are constants $a,b$ such
that for all $n,$%
\[
\frac{\alpha_{n+1}}{\alpha_{n}\beta_{n}^{\prime}}=a,\quad\frac{\alpha
_{n+1}B_{n}}{\beta_{n}^{\prime}}=b
\]

Since $\alpha_{n}\not =0$ for all $n,$ we see that $a\not =0$. These equalites
yield%
\begin{equation}
\beta_{n}^{\prime}=\frac{\alpha_{n+1}}{a\alpha_{n}},\quad B_{n}=\frac
{b}{a\alpha_{n}} \label{ac1}%
\end{equation}
with $\alpha_{n}$ unspecified. Under these assumptions, (\ref{rh}) folds to%
\begin{equation}
x_{n+2}=ax_{n+1}^{2}+bx_{n+1},\quad y_{n}=\frac{\alpha_{n}x_{n}}{x_{n+1}}
\label{hsc}%
\end{equation}

By a change of variables $r_{n}=x_{n+1}$ the first-order, autonomous core
equation above may be written as%
\begin{equation}
r_{n+1}=ar_{n}^{2}+br_{n},\quad r_{0}=x_{1}=\frac{\alpha_{0}x_{0}}{y_{0}}
\label{o1q}%
\end{equation}

If $b\not =0$ then the quadratic equation above exhibits complex behavior in
some invariant interval for a range of parameter values. This behavior for the
x-components occurs regardless of the choice of $\alpha_{n}$, and in
particular, when $\alpha_{n}=\alpha$ is constant, i.e., the autonomous case.

Equation (\ref{o1q}) is conjugate to the logistic equation $t_{n+1}%
=bt_{n}(1-t_{n})$ via the substitution $t_{n}=-ar_{n}/b$. It is well-known
that as $b$ goes from 3 to 4 the solutions of the logistic equation in the
invariant interval [0,1] undergo a familiar sequence of bifurcations. In
particular, the logistic equation has a period 3 solution when $b>3.83$ so
(\ref{o1q}) has periodic solutions of all possible periods due to the
Sharkovski ordering and exhibits chaos in [0,1] in the sense of Li and Yorke.
In fact, lower values of $b$ may be used for the onset
of chaos; e.g., $b>3.7$ where a snap-back repeller is born, or even 
$b\geq3.57$ that corresponds to the end of the
period-doubling cascade. These observations lead to the following result.

\begin{theorem}
\label{crh}Consider the system (\ref{rh}) subject to (\ref{rsc}), i.e., the
system
\begin{subequations}
\label{rhsc}%
\begin{align}
x_{n+1}  &  =\frac{\alpha_{n}x_{n}}{y_{n}}\\
y_{n+1}  &  =\frac{\beta_{n}^{\prime}y_{n}}{x_{n}+B_{n}y_{n}}%
\end{align}
Assume also that (\ref{ac1}) holds with $0<b<4$.

(a) If $(x_{0},y_{0})$ is an initial point such that
\end{subequations}
\begin{equation}
\frac{\alpha_{0}x_{0}}{y_{0}}\in\left(  -\frac{b}{a},0\right)  \quad\text{if
}a>0;\quad\frac{\alpha_{0}x_{0}}{y_{0}}\in\left(  0,-\frac{b}{a}\right)
\quad\text{if }a<0 \label{crh1}%
\end{equation}
then the following are true:

\qquad(i) The orbit $\{(x_{n},y_{n})\}$ is well-defined with $y_{n}=\alpha
_{n}x_{n}/x_{n+1}$ and $x_{n}\in(0,-b/a)$ if $a<0$, $x_{n}\in(-b/a,0)$ if
$a>0$ for all $n\geq0.$ Further, the orbit is bounded if $\{\alpha_{n}\}$ is bounded.

\qquad(ii) If $\lim_{n\rightarrow\infty}\alpha_{n}=\alpha\not =0$ and the
solution $\{r_{n}\}$ of (\ref{o1q}) converges to a $q$-cycle then the orbit
$\{(x_{n},y_{n})\}$ converges to a $q$-cycle.

\qquad(iii) If $\{\alpha_{n}\}$ converges to zero then the orbit
$\{(x_{n},y_{n})\}$ converges to a limit set that is contained in the x-axis.
If the solution $\{r_{n}\}$ of (\ref{o1q}) converges to a cycle or is chaotic
then the orbit has the same behavior but the limit set itself does not contain
a solution of (\ref{rhsc}).

\qquad(iv) If $\{\alpha_{n}\}$ converges to a $p$-cycle and the solution
$\{r_{n}\}$ of (\ref{o1q}) converges to a $q$-cycle then the orbit
$\{(x_{n},y_{n})\}$ converges to a cycle with period $\operatorname{lcm}(p,q)$.

\qquad(v) If $\{\alpha_{n}\}$ is bounded and the solution $\{r_{n}\}$ of
(\ref{o1q}) is chaotic (e.g., if $3.83<b<4)$ then the orbit $\{(x_{n}%
,y_{n})\}$ is chaotic.

(b) If $(x_{0},y_{0})$ is such that $\alpha_{0}x_{0}/y_{0}$ does not satisfy
(\ref{crh1}) and $\alpha_{0}x_{0}/y_{0}\not =0,\pm b/a$ then the orbit
$\{(x_{n},y_{n})\}$ is well-defined and unbounded.
\end{theorem}

\begin{proof}
(a) We prove the case $a<0$ here and leave out the analogous arguments for the
case $a>0.$

(i) Let $\alpha_{0}x_{0}/y_{0}=r_{0}\in\left(  0,-b/a\right)  .$ The critical
point of $\mu(r)=ar^{2}+br$ at $r=-b/2a$ yields the maximum value $\mu_{\max
}=-b^{2}/4a$. It follows that%
\[
0<-\frac{b^{3}}{4a}\left(  1-\frac{b}{4}\right)  =\mu(\mu_{\max})\leq
r_{n}\leq\mu_{\max}<-\frac{b}{a}%
\]
for all $n$ sufficiently large. In particular, $x_{n}=r_{n-1}$ does not
approach 0 so $y_{n}=\alpha_{n}x_{n}/x_{n+1}$ is well-defined. Further, since%
\[
\left\vert y_{n}\right\vert =\frac{|\alpha_{n}|x_{n}}{x_{n+1}}=\frac{r_{n-1}%
}{r_{n}}|\alpha_{n}|\leq\frac{\mu_{\max}}{\mu(\mu_{\max})}\left\vert
\alpha_{n}\right\vert =\frac{16}{b^{2}(4-b)}\left\vert \alpha_{n}\right\vert
\]
it follows that the orbit $\{(x_{n},y_{n})\}$ is bounded if $\{\alpha_{n}\}$ is.

(ii) Since $x_{n}=r_{n-1}$ for $n\geq1$ if $\{r_{n}\}$ converges to a
$q$-cycle in $\left(  0,-b/a\right)  $ then $\{x_{n}\}$ converges to the same
$q$-cycle (with a phase shift), say, $\lim_{n\rightarrow\infty}|x_{n}-\xi
_{n}|=0$ where $\{\xi_{n}\}$ is a $q$-cycle in the interval $[\mu(\mu_{\max
}),\mu_{\max}]$. Then $\xi_{n+q}/\xi_{n+1+q}=\xi_{n}/\xi_{n+1}$ for all $n$ so
$\{\xi_{n}/\xi_{n+1}\}$ has period $q$ and%
\begin{align*}
\left\vert \frac{x_{n}}{x_{n+1}}-\frac{\xi_{n}}{\xi_{n+1}}\right\vert  &
\leq\frac{1}{x_{n+1}\xi_{n+1}}\left(  \xi_{n+1}|x_{n}-\xi_{n}|+\xi_{n}%
|x_{n+1}-\xi_{n+1}|\right) \\
&  \leq\frac{\mu_{\max}}{\mu(\mu_{\max})^{2}}\left(  |x_{n}-\xi_{n}%
|+|x_{n+1}-\xi_{n+1}|\right)
\end{align*}

Thus $\left\{  x_{n}/x_{n+1}\right\}  $ converges to the periodic sequence
$\{\xi_{n}/\xi_{n+1}\}$ with period $q$. Since%
\begin{align*}
\left\vert y_{n}-\frac{\alpha\xi_{n}}{\xi_{n+1}}\right\vert  &  \leq\left\vert
\frac{\alpha_{n}x_{n}}{x_{n+1}}-\frac{\alpha_{n}\xi_{n}}{\xi_{n+1}}\right\vert
+\left\vert \frac{\alpha_{n}\xi_{n}}{\xi_{n+1}}-\frac{\alpha\xi_{n}}{\xi
_{n+1}}\right\vert \\
&  \leq|\alpha_{n}|\left\vert \frac{x_{n}}{x_{n+1}}-\frac{\xi_{n}}{\xi_{n+1}%
}\right\vert +|\alpha_{n}-\alpha|\frac{\mu_{\max}}{\mu(\mu_{\max})}%
\end{align*}
it follows that $\{y_{n}\}$ converges to the sequence $\{\alpha\xi_{n}%
/\xi_{n+1}\}$ which has period $q.$ Hence, the orbit $\{(x_{n},y_{n})\}$
converges to a sequence with period $q$.

(iii) By (i) above, $\mu(\mu_{\max})\leq x_{n}/x_{n+1}\leq\mu_{\max}$ so
$\lim_{n\rightarrow\infty}y_{n}=0$ and the limit set of $\{(x_{n},y_{n})\}$ is
contained in the x-axis. If $\{r_{n}\}$ converges to a cycle or is chaotic
then so is $\{x_{n}\}$ and the same behavior is exhibited by $\{(x_{n}%
,y_{n})\}$ as it approaches the x-axis. The limit set in the x-axis may be
finite or infinite depending on whether the limit of $\{x_{n}\}$ is periodic
or not. However, the limit set itself cannot be a solution of the system where
$y_{n}\not=0$ must hold for all $n$.

(iv) Suppose that $\{r_{n}\}$ converges to a $q$-cycle. Then $\{x_{n}\}$
converges to a $q$-cycle $\{\xi_{n}\}$ in the interval $[\mu(\mu_{\max}%
),\mu_{\max}].$ As in (ii), $\left\{  x_{n}/x_{n+1}\right\}  $ converges to
the periodic sequence $\{\xi_{n}/\xi_{n+1}\}$ with period $q.$ If
$\{\alpha_{n}\}$ converges to a sequence $\{\alpha_{n}^{\ast}\}$ of period $p$
then by Lemma \ref{per} $\{\alpha_{n}^{\ast}\xi_{n}/\xi_{n+1}\}$ has period
$\operatorname{lcm}(p,q)$ and
\begin{align*}
\left\vert y_{n}-\frac{\alpha_{n}^{\ast}\xi_{n}}{\xi_{n+1}}\right\vert  &
\leq\left\vert \frac{\alpha_{n}x_{n}}{x_{n+1}}-\frac{\alpha_{n}\xi_{n}}%
{\xi_{n+1}}\right\vert +\left\vert \frac{\alpha_{n}\xi_{n}}{\xi_{n+1}}%
-\frac{\alpha_{n}^{\ast}\xi_{n}}{\xi_{n+1}}\right\vert \\
&  \leq|\alpha_{n}|\left\vert \frac{x_{n}}{x_{n+1}}-\frac{\xi_{n}}{\xi_{n+1}%
}\right\vert +|\alpha_{n}-\alpha_{n}^{\ast}|\frac{\mu_{\max}}{\mu(\mu_{\max})}%
\end{align*}

Therefore, $\{y_{n}\}$ converges to the sequence $\{\alpha_{n}^{\ast}\xi
_{n}/\xi_{n+1}\}$ with period $\operatorname{lcm}(p,q).$ Hence, the orbit
$\{(x_{n},y_{n})\}$ converges to a sequence with period $\operatorname{lcm}%
(p,q)$.

(v) If $\{r_{n}\}$ is chaotic then so is $\{x_{n}\}.$ If $\{\alpha_{n}\}$ is
bounded then $\{y_{n}\}$ is also bounded since $x_{n}\in$ $[\mu(\mu_{\max
}),\mu_{\max}]$ for all $n\geq0.$ Therefore, regardless of the nature of the
behavior of $\{y_{n}\},$ the orbit $\{(x_{n},y_{n})\}$ is chaotic in the sense
of Definition \ref{chs}.

(b) If $(x_{0},y_{0})$ is such that $\alpha_{0}x_{0}/y_{0}$ does not satisfy
(\ref{crh1}) and $\alpha_{0}x_{0}/y_{0}\not =0,\pm b/a$ then the solution
$r_{n}$ of (\ref{o1q}) is unbounded. Thus the sequence $\{x_{n}\}$ is also
unbounded and it follows that the orbit $\{(x_{n},y_{n})\}$ is unbounded,
regardless of the nature of $\{y_{n}\}$.
\end{proof}

An example of a system to which the preceding result applies is the following%
\begin{align*}
x_{n+1}  &  =\frac{(-1)^{n}x_{n}}{y_{n}}\\
y_{n+1}  &  =\frac{-y_{n}}{ax_{n}+b(-1)^{n}y_{n}}%
\end{align*}
where $a,b$ are real numbers such that $a\not =0.$ Note that this system does
not converge to an autonomous system though it is periodic. A more comprehensive 
study of \ref{rh} in the future through folding and (\ref{rhc}) may reveal
additional interesting possibilities.

\begin{remark}
Certain exceptional solutions of (\ref{rhsc}) cannot be derived from the
folding. In particular, if $x_{0}=0$ and $y_{0}\not =0$ then $x_{n}=0$ for all
$n$ so that $y_{n}=\beta_{n}^{\prime}/B_{n}=\alpha_{n+1}/b.$ Thus the sequence
$\{(0,\alpha_{n+1}/b)\}$ is an orbit of (\ref{rhsc}) that cannot be obtained
from the ratio $\alpha_{n}x_{n}/x_{n+1}$ in the passive equation. This orbit
is unstable if $b>1$, a parameter range for which 0 is unstable in (\ref{o1q})
but if $0<b\leq1$ then it attracts all orbits of the system with $\alpha
_{0}x_{0}/y_{0}$ as given in (\ref{crh1}).
\end{remark}

\section{An inverse problem and more rational systems\label{inv}}

Folding a given nonlinear system into a higher order equation does not always
simplify the study of solutions. From a practical point of view, a significant
gain in terms of simplifying the analysis of solutions is desirable. This was
the case in the previous section where the folding had a one-dimensional
structure. In this section, we determine and study classes of difference
systems that fold to equations of order 1 or 2 with known properties. We start
with one of the two equations of the system, say, the one given by $f$ along
with a known function $\phi$ that defines a second-order equation with desired
properties. Then a function $g$ is determined with the property that the
system with components $f$ and $g$ folds to an equation of order 2 defined by
$\phi.$ 

This process is indeed an inverse of folding in the sense that the resulting system
with $f$ and $g$ is a (non-standard) \textit{unfolding} of the equation of 
order 2 that is defined by the function $\phi$. In the autonomous case, if 
$f(u,v)=v$ then $g=\phi$ and we obtain a standard unfolding.

Using a rational function $f$ the above unfolding process leads, in particular, to 
a rediscovery of the rational system discussed in the previous section as a 
non-standard unfolding of the first-order logistic equation. By unfolding other
first or second order difference equations in this way, we discover other rational 
systems that are not homogeneous but which can still be analyzed using the same method.

Suppose that a function $f$ satisfies Definition \ref{smsl}. By (\ref{se2})
the following
\[
f(n+1,w,g(n,u,h(n,u,w)))=\phi(n,u,w)
\]
is a function of $n,u,w.$ Since $f$ is semi-invertible, once again from
Definition \ref{smsl} we obtain%
\begin{equation}
g(n,u,h(n,u,w))=h(n+1,w,\phi(n,u,w)) \label{idc}%
\end{equation}

Now, suppose that $\phi(n,u,w)$ is prescribed on a set $\mathbb{N}%
_{0}\mathbb{\times}M^{\prime}$ where $M^{\prime}\subset S\times S$ and we seek
$g$ that satisfies (\ref{idc}). Assume that a subset $M$ of $D$ exists with
the property that $f(\mathbb{N}_{0}\mathbb{\times}M)\times\phi(\mathbb{N}%
_{0}\mathbb{\times}M^{\prime})\subset M^{\prime}.$ For $(n,u,v)\in
\mathbb{N}_{0}\mathbb{\times}M$ define%
\begin{equation}
g(n,u,v)=h(n+1,f(n,u,v),\phi(n,u,f(n,u,v))) \label{idc1}%
\end{equation}

In particular, if $v\in h(\mathbb{N}_{0}\mathbb{\times}M^{\prime})$ then $g$
above satisfies (\ref{idc1}). These observations establish the following result.

\begin{theorem}
\label{idn}Let $f$ be a semi-invertible function with $h$ given by Definition
\ref{smsl}. Further, let $\phi$ be a given function on $\mathbb{N}%
_{0}\mathbb{\times}M^{\prime}$. If $g$ is given by (\ref{idc1}) then
(\ref{s1}) folds to the difference equation $s_{n+2}=\phi(n,s_{n},s_{n+1})$
plus a passive equation.
\end{theorem}

In separable cases, explicit expressions are possible with the aid of
(\ref{spsh}). Note that semilinear systems are included in the next result.

\begin{corollary}
Let $(G,\ast)$ be a nontrivial group and $f(n,u,v)=f_{1}(n,u)\ast f_{2}(n,v)$
be separable on $G\times G$ with $f_{2}$ a bijection. If $\phi$ is a given
function on $\mathbb{N}_{0}\mathbb{\times}G\times G$ and $g$ is given by
\[
g(n,u,v)=f_{2}^{-1}(n+1,\left[  f_{1}(n+1,f_{1}(n,u)\ast f_{2}(n,v))\right]
^{-1}\ast\phi(n,u,f_{1}(n,u)\ast f_{2}(n,v)))
\]
then (\ref{s1}) folds to the difference equation $s_{n+2}=\phi(n,s_{n}%
,s_{n+1})$ plus a passive equation.
\end{corollary}

The next result yields a class of systems that actually reduce to first-order
difference equations.

\begin{corollary}
\label{o1}Assume that $f,h$ satisfy the hypotheses of Theorem \ref{idn} and
let $\phi(n,\cdot)$ be a function of one variable for each $n$. If%
\[
g(n,u,v)=h(n+1,f(n,u,v),\phi(n,f(n,u,v)))
\]
then (\ref{s1}) folds to the difference equation $s_{n+2}=\phi(n,s_{n+1})$
with order 1 plus a passive equation.
\end{corollary}

We may use the above corollary to rediscover the rational system discussed in
the previous section. Let $f(n,u,v)=\alpha_{n}u/v$ as in (\ref{rha}) subject
to (\ref{rsc}). With $\phi(n,u,w)=aw^{2}+bw$ that defines (\ref{hsc}) we
obtain, using Corollary \ref{o1}%
\[
g(n,u,v)=\frac{\alpha_{n+1}\alpha_{n}u}{v[a(\alpha_{n}u)^{2}/v^{2}+ba_{n}%
u/v]}=\frac{\alpha_{n+1}v}{a\alpha_{n}u+bv}=\frac{(\alpha_{n+1}/a\alpha_{n}%
)v}{u+(b/a\alpha_{n})v}%
\]

Using the substitutions (\ref{ac1}) we obtain the homogeneous system
(\ref{rhsc}). The next result yields a class of systems that fold to second-order
difference equations whose orbits are determined by first-order equations.

\begin{corollary}
\label{cdn}Assume that $f,h$ satisfy the hypotheses of Theorem \ref{idn} and
let $\phi(n,\cdot)$ be a function of one variable for each $n.$ If%
\begin{equation}
g(n,u,v)=h(n+1,f(n,u,v),\phi(n,u)) \label{ief}%
\end{equation}
then (\ref{s1}) folds to the difference equation $s_{n+2}=\phi(n,s_{n})$ whose
even terms and odd terms are (separately) solutions of the first-order
equation%
\begin{equation}
r_{n+1}=\phi(n,r_{n}). \label{efo1}%
\end{equation}
i.e., $s_{2k}=\phi(2k-2,s_{2k-2})$ and $s_{2k+1}=\phi(2k-1,s_{2k-1})$ for all
$k\geq1$, $s_{0}=x_{0}$ and $s_{1}=f(0,x_{0},y_{0}).$
\end{corollary}

As an application, consider the function $f(n,u,v)=\alpha_{n}u/v$ again but
now with $\phi(n,u,w)=au^{2}+bu.$ Then Corollary \ref{cdn} yields%
\[
g(n,u,v)=\frac{\alpha_{n+1}\alpha_{n}u}{v(au^{2}+bu)}=\frac{\alpha_{n+1}%
\alpha_{n}}{v(au+b)}%
\]
which results in the system
\begin{subequations}
\label{rnh}%
\begin{align}
x_{n+1}  &  =\frac{\alpha_{n}x_{n}}{y_{n}}\\
y_{n+1}  &  =\frac{\alpha_{n}\alpha_{n+1}}{(ax_{n}+b)y_{n}}%
\end{align}

The core of its folding is $s_{n+2}=as_{n}^{2}+bs_{n}$ a second-order equation
of type seen in Corollary \ref{cdn}. The even- and odd-indexed terms are
generated by a conjugate of the logistic map so an analysis similar to that of
the previous section may be carried out for the rational system (\ref{rnh}).

The fact that, despite similarities, the folding of (\ref{rnh}) has order 2
whereas that of (\ref{rhsc}) has order 1 has some interesting consequences
about the corresponding systems and their orbits. To be more precise, consider
the autonomous version of (\ref{rhsc}) with $\alpha_{n}=\alpha$ for all $n$
i.e., the system%

\end{subequations}
\begin{subequations}
\label{mhs}%
\begin{align}
x_{n+1} &  =\frac{\alpha x_{n}}{y_{n}}\\
y_{n+1} &  =\frac{\beta y_{n}}{x_{n}+\gamma y_{n}}%
\end{align}
where $\beta=1/a$ and $\gamma=b/a\alpha$ which we compare to the autonomous
version of (\ref{rnh})
\end{subequations}
\begin{subequations}
\label{coch}%
\begin{align}
x_{n+1} &  =\frac{\alpha x_{n}}{y_{n}}\\
y_{n+1} &  =\frac{\beta}{(x_{n}+\gamma)y_{n}}%
\end{align}
with $\alpha_{n}=\alpha$, $\beta=\alpha^{2}/a$ and $\gamma=b/a.$ 

The graph of
a single typical orbit of (\ref{mhs}) that satisfies the hypotheses in Part (v) 
of Theorem \ref{crh} appears in Figure 1. 

\centerline{\includegraphics[width=3.78in,height=3.53in,keepaspectratio]{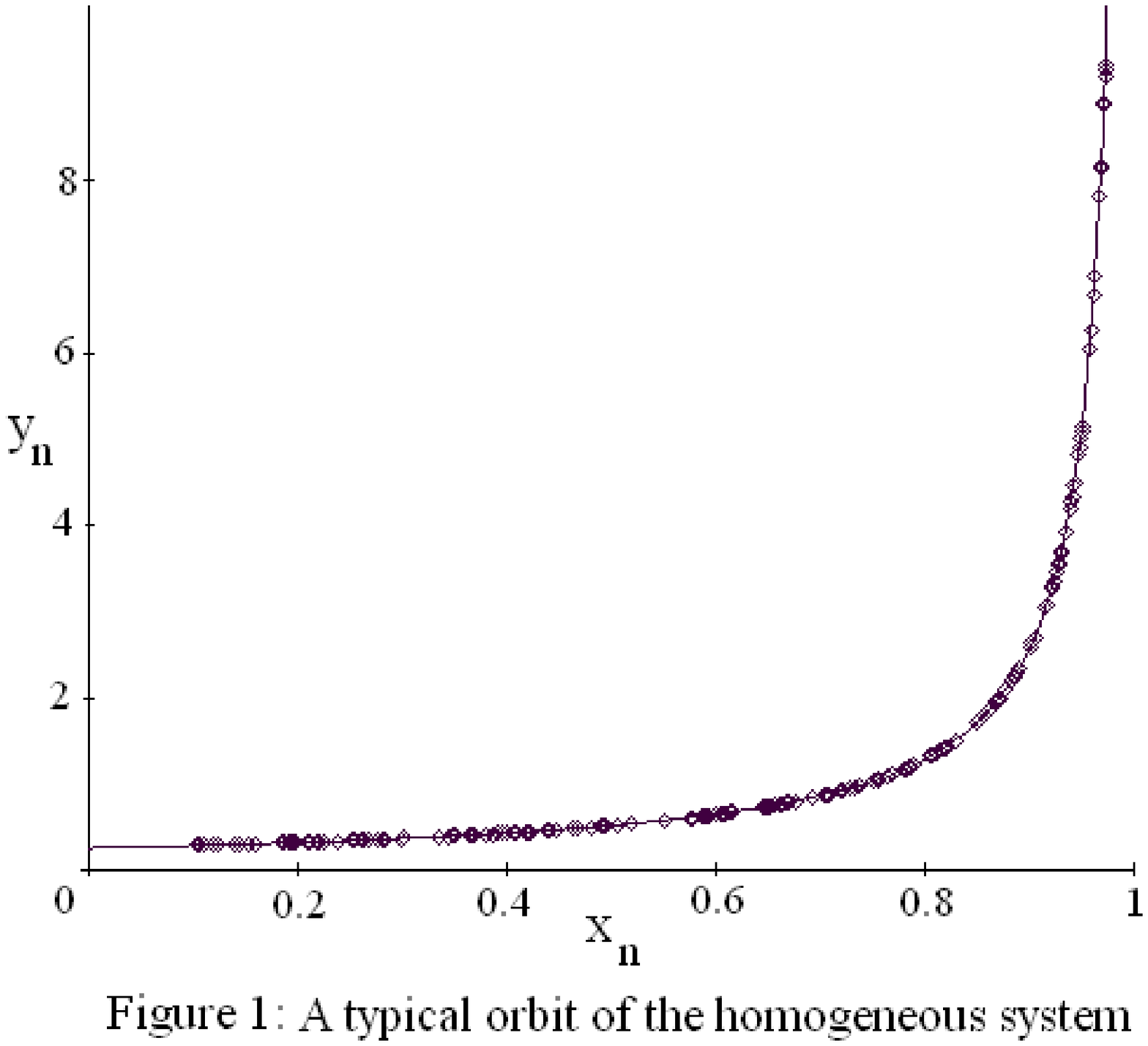}}

We note that the chaotic orbit is in the positive quadrant since $y_{n}>0$ for all $n$ in this case. The one dimensional manifold that contains the orbit is the curve
$y=\alpha/(ax+b)$ which is calculated using (\ref{hsc}) as follows%
\end{subequations}
\[
y_{n}=\frac{\alpha x_{n}}{x_{n+1}}=\frac{\alpha x_{n}}{ax_{n}^{2}+bx_{n}%
}=\frac{\alpha}{ax_{n}+b}.
\]

The graph of a single typical orbit of (\ref{coch}) is shown in Figure 2. The
spread of the orbit in the plane reflects the higher order of the folding in this case.

\centerline{\includegraphics[width=3.69in,height=3.53in,keepaspectratio]{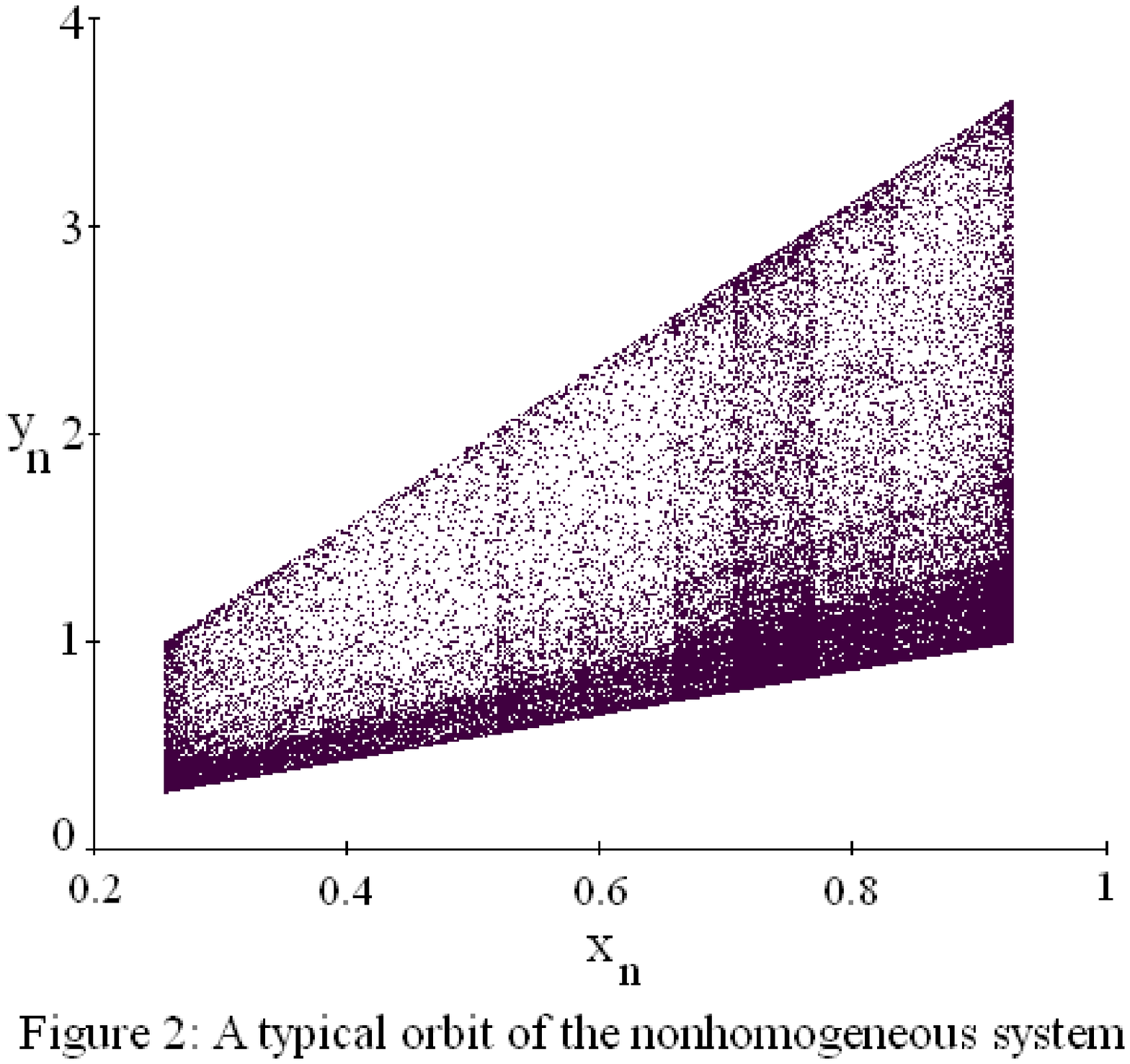}}

Other facts worth mentioning with regard to (\ref{mhs}) and (\ref{coch}) are
that the latter is not homogeneous and semiconjugacy to a known map is not
known for it. Further, the second equation of (\ref{coch}) is not linear-fractional.

The next result concerns systems that fold to autonomous affine
equations of order 2.

\begin{corollary}
\label{lin}Assume that $f,h$ satisfy the hypotheses of Theorem \ref{idn} and
let $\phi(u,w)=au+bw+c$ be an affine function where $|a|+|b|>0$. If%
\[
g(n,u,v)=h(n+1,f(n,u,v),au+bf(n,u,v)+c)
\]
then (\ref{s1}) folds to the difference equation $s_{n+2}=as_{n}+bs_{n+1}+c$
plus a passive equation.
\end{corollary}

As an application of the above corollary, consider $f(n,u,v)=\alpha_{n}u/v$
together with $\phi(u,w)=au+bw+c.$ Then by Corollary \ref{lin}%
\[
g(n,u,v)=\frac{\alpha_{n+1}\alpha_{n}u}{v[au+b(\alpha_{n}u/v)+c]}=\frac
{\alpha_{n+1}\alpha_{n}u}{auv+cv+b\alpha_{n}}%
\]
corresponding to the following rational system
\begin{subequations}
\label{lna}%
\begin{align}
x_{n+1}  &  =\frac{\alpha_{n}x_{n}}{y_{n}}\\
y_{n+1}  &  =\frac{\alpha_{n}\alpha_{n+1}x_{n}}{\alpha_{n}bx_{n}%
+(ax_{n}+c)y_{n}}%
\end{align}

In the special case where $\alpha_{n}=\alpha$ is a constant and $a=0$ the
above system takes the form%

\end{subequations}
\begin{subequations}
\label{lah}%
\begin{align}
x_{n+1} &  =\frac{\alpha x_{n}}{y_{n}}\\
y_{n+1} &  =\frac{\beta x_{n}}{x_{n}+\gamma y_{n}}%
\end{align}
\end{subequations}
where $\beta=\alpha/b$ and $\gamma=c/\alpha b$. This homogeneous system, which
folds to the affine first-order equation (\ref{ao1}) does not generate complex
behavior, in contrast to (\ref{mhs}), which is also homogeneous. On the other hand, 
if $b=0$ and
$\alpha_{n}=\alpha$ then (\ref{lna}) reduces to the autonomous system
\begin{subequations}
\label{lnh}%
\begin{align}
x_{n+1} &  =\frac{\alpha x_{n}}{y_{n}}\\
y_{n+1} &  =\frac{\beta x_{n}}{(x_{n}+\gamma)y_{n}}%
\end{align}
\end{subequations}
where $\beta$ $=\alpha^{2}/a$ and $\gamma=c/a.$ This system may be compared to
(\ref{coch}) since (\ref{lnh}) is not homogeneous, semiconjugacy to a known
map is not known for it and the second equation of the system is not
linear-fractional. But in contrast to (\ref{coch}), system (\ref{lnh}) folds
to a second-order affine difference equation and thus, cannot generate complex
behavior. In fact, general formulas for the orbits of (\ref{lah}) and
(\ref{lnh}) can be easily obtained in closed form if desired.



\end{document}